\documentclass[english]{amsart}

\usepackage{babel}
\usepackage{amstext}
\usepackage{amsmath}
\usepackage{amscd}
\usepackage{amsfonts}
\usepackage{latexsym}
\usepackage{ifthen}
\newcommand{\lra}{\longrightarrow}

\newcommand\sO{{\mathcal O}}

\newcommand\bC{{\mathbb C}}
\newcommand\bQ{{\mathbb Q}}

\newcommand\bP{{\mathbb P}}

\newcounter{lemma}

\usepackage{enumerate}

\theoremstyle{plain} %text of this environment is typesetted in italics
\newtheorem{theorem}{\noindent\bf Theorem}[section]
\newtheorem{lemma}[theorem]{\noindent\bf Lemma}
\newtheorem{corollary}[theorem]{\noindent\bf Corollary}

\theoremstyle{definition}
\newtheorem{remark}[theorem]{\noindent\bf Remark}

\newtheorem{question}[theorem]{\noindent\bf Question}

\title[saturation] 
{A simple remark on a flat projective morphism with a Calabi-Yau fiber}
\author{Keiji Oguiso}
\dedicatory{Dedicated to Professor Dr. Fabrizio Catanese 
on the occasion of his sixtieth birthday}


\begin{document}

\begin{abstract} If a K3 surface is a fiber of a flat projective 
morphisms over a connected noetherian scheme over the complex number field, then any smooth connected fiber is also a K3 surface. 
Observing this, Professor Nam-Hoon Lee asked if the same is true 
for higher dimensional Calabi-Yau fibers. We shall give an explicit negative 
answer to his question in each dimension greater than or equal to three 
as well as a proof of his initial observation.  
\end{abstract}
\maketitle
\tableofcontents
\section{Introduction - Background and Main Results}
\noindent
%(In the text [...] are something not completely sure for me or
%missing references).

Throughout this paper, we shall work in the category of noetherian schemes 
over the complex number field $\bC$. By a point, we mean a closed point 
and by a space, we mean a noetherian scheme.

Let $f : X \lra B$ be a flat projective morphism 
over a connected space $B$. We denote the scheme theoretic 
fiber $f^{-1}(b)$ by $X_b$. 

It is classically known that if $X_b$ is a smooth curve of genus 
$g$ and $X_s$ is also a smooth irreducible fiber, then $X_s$ is 
also a curve of 
the same genus $g$. 

Nam-Hoon Lee observed the following:

\begin{theorem}\label{K3} If $X_b$ is a K3 surface for a point 
$b \in B$ and $X_s$ is smooth, irreducible 
for a point $s \in B$, then $X_s$ 
is also a K3 surface.
\end{theorem}

Here a {\it K3 surface} is a smooth projective surface with 
trivial canonical bundle 
$\sO_S(K_S) \simeq \sO_S$, whose underlying analytic space 
is simply connected. See also Theorem (\ref{K3S}).

He then asked the following:

\begin{question}\label{nh} Suppose that $X_b$ is a Calabi-Yau manifold for a point $b \in B$ and $X_s$ is smooth, irreducible for a point $s \in B$. 
Is then $X_s$ a Calabi-Yau manifold? 
\end{question}

Here, by {\it a Calabi-Yau manifold}, we mean a smooth projective irreducible 
variety $M$ with trivial canonical bundle $\sO_M(K_M) \simeq \sO_M$, whose underlying analytic space is simply connected. Note 
that the Kodaira dimension $\kappa (M)$ of $M$ is $0$ and 
the irregularity $h^1(\sO_M)$ of $M$ is also $0$.

The aim of this note is to give a negative answer to Question (\ref{nh}):

\begin{theorem}\label{cy}

(1) For each integer $d \ge 3$, $d \not= 4$, 
there are a flat projective morphism 
$f : X \lra B$ over a connected space $B$ 
and points $b, s \in B$ such that $X_b$ is a Calabi-Yau $d$-fold 
and $X_s$ is a smooth, irreducible $d$-fold with 
$\kappa (X_s) = -\infty$ and $h^1(\sO_{X_s}) = 1$;

(2) There are a flat projective morphism 
$f : X \lra B$ over a connected space $B$ 
and points $b, s \in B$ such that $X_b$ is a Calabi-Yau $4$-fold 
and $X_s$ is a smooth, irreducible $4$-fold with 
$\kappa (X_s) = -\infty$ (and $h^1(\sO_{X_s}) = 0$).
\end{theorem} 

Note that in Theorem (\ref{cy}), the base space $B$ is necessarily reducible 
and two points $s$ and $b$ have to belong to different irreducible components of $B$. This follows from \cite{BHPV04}, Chapter IV, Proposition (4.4) with obvious necessary modification on dimension. This theorem in particular says that the invariance of pluri-genera does not hold in arbitrary flat projective deformations. This is pointed out to me by Professor Shing-Tung Yau. See also Corollary (\ref{kappa}). 

Our theorem is nothing but a simple application of Hartshorne's result 
(\cite{Ha66}, main result, see also \cite{PS05} for a shorter proof 
in characteristic $0$ 
and Theorem (\ref{ha}) in Section 2) and Xu's result (\cite{Xu95}, Theorem, 
see also Theorem (\ref{xu}) in Section 2). 
\par
\vskip 4pt
\noindent {\it Acknowledgements.} I would like to express my thank to Professor Nam-Hoon Lee for his question and for encouraging me to include his original observation. I would like to express my thank to Professor Shing-Tung Yau for his interest in this work and for inviting me to Haravard with full financial support, where the expanded version is completed. Last but not least at all, it is my honor to dedicate this note to Professor Dr. Fabrizio Catanese on the occasion of his sixtieth birthday, from whom I learned much about mathematics and other things since I met him in 1993 in Bonn.

\section{Ingredients}
\noindent
By a {\it polarized manifold} $(M, H)$, we mean a pair of a smooth projective 
irreducible variety $M$ and an ample line bundle $H$ on $M$. We denote the Hilbert polynomial 
$$\chi(M, nH) = \sum_{k} (-1)^k h^k(M, nH)$$ 
simply by $P(M, H)$. 
 
As in \cite{Le06}, \cite{LO08}, our proof is based on the fundamental theorem 
of Hartshorne (\cite{Ha66}, main result):
\begin{theorem}\label{ha} The Hilbert scheme ${\rm Hilb}_{\bP^N}^{P(n)}$ 
of $\bP^N$ with a fixed Hilbert polynomial $P(n)$ is connected. 
The universal family $u : U_{\bP^N}^{P(n)} \lra {\rm Hilb}_{\bP^N}^{P(n)}$ 
is then a flat projective morphism over the connected space.
\end{theorem}

Note that for a given two polarized manifolds 
$(M_j, H_j)$ ($j = 1,2$), there is a large common integer $r_0$ such that 
$rH_j$ are very ample and $h^i(rH_j) = 0$ for all $i > 0$ and $r \ge r_0$. 
Thus, if $P(M_1, H_1) = P(M_2, H_2)$, then $h^0(M_j, rH_j)$ 
are the same positive integer, say $N(r)+1$. Then $M_j$ are embedded, 
by $\vert r_0H_j \vert$, 
into the same projective space $\bP^{N(r_0)}$ in which 
$r_0H_j$ is the restriction of the hyperplane bundle $\sO_{\bP^N}(1)$ 
to $M_j$ and $P(M_1, r_0H_1) = P(M_2, r_0H_2)$ (by the choice of $r_0$). 
Thus, by Theorem (\ref{ha}), we have:
\begin{corollary}\label{hac}
Given two smooth projective irreducible 
varieties $M_j$ ($j =1,2$), the following 
(1) and (2) are equivalent:
\begin{enumerate}
\item $M_j$ are fibers of a flat projective morphism 
$f : X \lra B$ 
over some connected space $B$.
\item there are ample line bundles $H_j$ on $M_j$ such that 
$P(M_1, H_1) = P(M_2, H_2)$.
\end{enumerate}
\end{corollary}
This corollary gives us some restriction of smooth irreducible 
fibers of a flat projective morphism with a Calabi-Yau fiber (even if very bad 
singular fibers are allowed):
\begin{corollary}\label{kappa}
Let $f : X \lra B$ be a flat projective morphism over a connected space $B$ 
and $b, s \in B$ be points such that $X_b$ is a Calabi-Yau $d$-fold 
and $X_s$ is a smooth, irreducible $d$-fold. Then 
$\kappa (X_s)$ is either $0$ or $-\infty$. Moreover if 
$\kappa (X_s) = 0$, 
then the canonical divisor $K_{X_s}$ is $\bQ$-linearly equivalent to 
$0$. 
\end{corollary}

\begin{proof} Let $L$ be an $f$-ample line bundle on $X$. 
Consider the Hilbert polynomial $P(X_s, L_s)$ of $(X_s, L_s)$ as a function 
of $s$ (valued in a polynomial ring). Since $B$ is connected and $P(X_s, L_s)$ is a locally constant function of $s \in B$,  it follows that $P(X_s, L_s)$ 
is independent of 
$s \in B$. Since $X_b$ is a Calabi-Yau $d$-fold, the coefficient of the term 
$n^{d-1}$ is then $0$. Thus, by the Riemann-Roch theorem, 
$(L_s^{d-1}.K_{X_s}) = 0$ as well for smooth irreducible $X_s$. Since $L_s$ is ample on $X_s$, there is then no positive integer 
$m$ such that $mK_{X_s}$ is linearly equivalent to a non-zero effective 
divisor. This proves the result.
\end{proof}

Besides Theorem (\ref{ha}), in order to construct examples in 
Theorems (\ref{K3S}), (\ref{cy}), the following theorem due to Xu 
(\cite{Xu95} Theorem) is also very useful:

\begin{theorem}\label{xu} Let $\pi : S_k \lra \bP^2$ be 
the blow up of $\bP^2$ at $k$ generic closed 
points $p_i$ ($1 \le i \le k$) and $p$ be an integer. 
We denote $\pi^*\sO_{\bP^2}(1)$ by $h$ and 
the exceptional divisor over $p_i$ by $e_i$. Then the line bundle
$$\ell_{p, k} := ph - \sum_{i=1}^{k} e_i$$
is ample on $S_k$, provided that $p > 2$ and $p^2 > k >0$. 
\end{theorem}

\section{Proof of Theorem (\ref{K3})}
\noindent

We shall prove Theorem (\ref{K3}) in the following slightly 
expanded form:

\begin{theorem}\label{K3S} 

(1) Let $f : X \lra B$ a flat projective morphism 
over a connected space $B$. 
Assume that $X_b$ is a K3 surface 
and $X_s$ is smooth, irreducible. Then $X_s$ 
is also a K3 surface. 

(2) There are a flat projective morphism 
$f : X \lra B$ over a connected space $B$ and points $b, s \in B$ such that 
$X_b$ is a smooth Enriques surface and $X_s$ is a smooth rational surface.
\end{theorem}

\begin{proof} First, we shall show the assertion (1). By Corollary (\ref{kappa}), it follows that either $\kappa(X_s) = -\infty$ or 
$K_{X_s}$ is $\bQ$-linearly 
equivalent to $0$. We have 
$$\chi(\sO_{X_s}) = \chi(\sO_{X_b}) = 2$$
by the flatness and by the fact that $X_b$ is a K3 surface. 
If $\kappa(X_s) = -\infty$, then 
$\chi(\sO_{X_s}) \le 1$ by the classification of surfaces. Thus, $K_{X_s}$ is $\bQ$-linearly equivalent to $0$ (in particular, minimal). Then again by 
$\chi(\sO_{X_s}) = 2$ and by the classification of surfaces, $X_s$ is a K3 surface. 

Let us show the assertion (2). Let $p$, $k$ be integers such that 
$$p \ge 4\,\, {\rm and}\,\, k = 3p\, .$$
Let $M$ a generic Enriques surface. 
Then, $M$ has no smooth rational curve by \cite{BP83} Proposition (2.8), 
and admits an elliptic fibration $\varphi : M \lra \bP^1$ 
such that each fiber is irreducible, there are exactly two multiple fibers 
$2F$ and $2F'$ and a multi-section $C$ of $\varphi$ of degree $2$, where $F$, $F'$ and $C$ are irreducible curves with $(F^2) = (C^2) = 0$ 
(See \cite{BHPV04}, Chapter VIII, Section 17). Note that $(F.C) =1$ 
by definition. 
Since the fibers of $\varphi$ are irreducible, by the Nakai criterion, the line bundle
$$f_{m} := F + mC$$ 
is ample for each integer $m \ge 1$. In particular it is ample for 
$m = p(p-3)/2$ (Note here that $p(p-3)$ is always even). By the Riemann-Roch theorem, we have
$$P(M, f_{p(p-3)/2}) = \frac{((F + p(p-3)C/2)^2)}{2}n^2 + 
\chi(\sO_M) = \frac{p(p-3)}{2}n^2 + 1
\, .$$ 
Let $S_k$ and $\ell_{p, k}$ be as in Theorem (\ref{xu}). By $p \ge 4$ and $k = 3p$, it follows that $p >2$ and $p^2 >k >0$. Thus, by Theorem (\ref{xu}), 
the line bundle $\ell_{p, k}$ is ample on $S_k$. Note that 
$$K_{S_k} = -3h + \sum_{i=1}^{k} e_i$$ 
by the adjunction formula. Calculate that
$$(\ell_{p,k}^2) = p^2 - k = p(p-3)\,\, {\rm and}\,\, (\ell_{p,k}.K_{S_k}) = -3p + k = 0\, .$$
Then, by the Riemann-Roch theorem, we get
$$P(S_k, \ell_{p,k}) = \frac{(\ell_{p, k}^2)}{2}n^2 - \frac{(\ell_{p,k}.K_{S_k})}{2}n + \chi(\sO_{S_k}) = 
\frac{p(p-3)}{2}n^2 + 1\, .$$
Thus $P(S_k, \ell_{p,k}) = P(M, f_{p(p-3)/2})$ and the result follows from Corollary (\ref{hac}). 
\end{proof}

\section{Proof of Theorem (\ref{cy})}
\noindent

In this section, we shall prove Theorem (\ref{cy}). 

First, we shall show Theorem (\ref{cy})(1) for $d=3$. 

By Corollary (\ref{hac}), it suffices to find a polarized Calabi-Yau 
threefold $(M, H)$ and a polarized smooth threefold $(V, L)$ 
such that $\kappa(V) = -\infty$, $h^1(\sO_V) = 1$ and $P(M, H) 
= P(V, L)$. 

Let $\varphi_i : R_i \lra \bP^1$ ($i = 1, 2$) be two generic 
relatively minimal rational surfaces with section $O_i$. Then, by 
\cite{Sc88} Section 2, 
the fiber product $M = R_1 \times_{\bP^1} R_2$ is a Calabi-Yau 
threefold. Under the identification $R_1 = p_2^*O_2$ and $R_2 = p_1^*O_1$, 
where $p_i : M \lra R_i$ are the natural projections, we regard 
$R_i$ as subvarieties of $M$. Let $F$ be a general fiber of 
$\varphi_1 \circ p_1 = \varphi_2 \circ p_2 : M \lra \bP^1$. 
The next Lemma is observed in \cite{Og91} Proposition-Example (0.14) and Claim 
(0.15):

\begin{lemma}\label{cyn} Consider the line bundle 
$H_m = R_1 + R_2 + mF$ on $M$, where $m$ be 
an integer. Then 

(1) $(H_m^3) = 6(m-1)$ and $(H_m.c_2(X)) = 12$. 

(2) $H_m$ is ample if $m \ge 3$. 
\end{lemma}
\begin{remark} Strictly speaking, the proof (2) in \cite{Og91} 
seems slightly imcomplete. In fact, \cite{Og91} checks only $(H_m.C) >0$ 
for any irreducible curve $C$ on $M$. However, this is enough to obtain 
(2). Indeed, $H_m$ is then nef and also big by (1). Since $K_X = 0$, 
the complete linear system $\vert NH_m \vert$, $N$ being a large integer, 
defines a birational morphism 
onto the image, which is normal. This is a consequence 
of the base point free theorem 
\cite{Ka84} Theorem 2.6 together with the fact 
that the Stein factorization map 
is finite. Thus, again by $(H_m.C) >0$, it follows that $NH_m$ is 
very ample by the Zariski main Theorem. 
\end{remark}

Thus, by the Riemann-Roch theorem for a Calabi-Yau threefold, we obtain
\begin{corollary}\label{cyr} 
$$P(M, H_m) = \frac{(H_m^3)}{6}n^3 + \frac{(H_m.c_2(M))}{12} = (m-1)n^3 + 2n\, 
.$$
\end{corollary}

Let $S_k$ and $\ell_{p, k}$ be as in Theorem (\ref{xu}). 
Take an elliptic curve $E$ and consider 
the product manifold 
$$V_k = S_k \times E\, .$$ 
Since $S_k$ is a rational surface and $E$ is an elliptic curve, 
it follows that 
$$\kappa(V_k) = -\infty\,\, {\rm and}\,\, h^1(\sO_{V_k}) = 1\, .$$ 
Let us consider the following line bundle on $V_k$:
$$L_{p,k} = p_1^*\ell_{p, k} + p_2^*\eta\, ,$$
where $\eta$ is a line bundle 
of degree $2$ on $E$ and $p_1 : V_k \lra S_k$ and $p_2 : V_k \lra E$ 
are the natural projections. As in the proof of Theorem (\ref{K3S})(2), 
we choose and fix (in what follows) integers $p$ and $k$ such that
$$p \ge 4\,\, {\rm and}\,\, k = 3p\, .$$
\begin{lemma}\label{alt} The line bundle $L_{p, k}$ 
is ample and
$$P(V_k, L_{p, k}) = p(p-3)n^2 + 2n\, .$$ 
\end{lemma}
\begin{proof} By Theorem (\ref{xu}), $\ell_{p,k}$ is ample. $\eta$ is clearly 
ample on $E$. Thus $L_{p,k}$ is ample as well. By the Riemann-Roch theorem and the Kunneth formula, we have
$$P(V_k, L_{p,k}) = P(S_k, \ell_{p, k}) \cdot P(E, \eta)$$
$$= (\frac{p(p-3)}{2}n^2 + 1)\cdot 2n = p(p-3)n^3 + 2n\, .$$
\end{proof}
Thus, by putting 
$$m = p(p-3) + 1$$ 
in Corollary (\ref{cyr}), we obtain that
$$P(V_{3p}, L_{p,k}) = P(M, H_m)\, .$$ 
Thus $M$ and $V = V_{k}$ ($k = 3p$ with $p \ge 4$) are such manifolds as in 
Theorem (\ref{cy}) (1) for $d=3$. 

Next we shall show Theorem (\ref{cy}) (1) for $d \ge 5$. We already obtained 
a flat projective morphism $f : X \lra B$ and points $b, s \in B$ 
such that $X_b = M$ and $X_s = V$, where $M$ and $V$ are the manifolds found 
above. Let $S$ be any projective K3 surface. 

Assume that $d \ge 5$ is odd. Let us consider 
the product 
$$Y := X \times S^{(d-3)/2}\, .$$ 
Then the naturally induced morphism $f_Y : Y \lra B$ is a flat projective 
morphism such that $Y_b = M \times S^{(d-3)/2}$ and $Y_s = V 
\times S^{(d-3)/2}$. This proves Theorem (\ref{cy}) (1) for all odd 
$d \ge 5$.

Assume that $d \ge 5$ is even. Let us consider the product 
$$Z := X \times M \times S^{(d-6)/2}\, .$$ 
Then the naturally induced morphism $f_Z : Z \lra B$ is a flat projective 
morphism such that $Z_b = M \times M \times S^{(d-6)/2}$ and 
$Z_s = V \times M \times S^{(d-6)/2}$. This proves Theorem 
(\ref{cy}) (1) for all even $d \ge 5$. 

Let us show Theorem (\ref{cy}) (2). 

Let $(S, \ell)$ be a polarized K3 surface such that $(\ell^2) = 2r$. 
As well-known, for each $r > 0$, such a polarized K3 surface exists. 
By the Riemann-Roch theorem, we have 
$$P(S, \ell) = rn^2 + 2\, .$$
Let $p$ and $k$ be integers such that $p \ge 4$ and $k = 3p$. Let 
$(S_k, \ell_{p, k})$ be the polarized rational surface in Theorem 
(\ref{xu}). Let us consider the product $4$-fold and its ample line bundle
$$T_k := S \times S_k\,\, ,\,\, L_k := p_1^*\ell 
\otimes p_2^*\ell_{p, k}$$
where $p_1 : T_k \lra S$ and $p_2 : T_k \lra S_k$ are the natural projections. 
Then, $\kappa(T_k) = -\infty$. By the Kunneth formula, we have
$$P(T_k, L_k) = P(S, \ell) \cdot P(S_k, \ell_{p, k})$$ 
$$= (rn^2 + 2)\cdot(\frac{p^2 -3p}{2}n^2 + 1) 
= \frac{(p^2-3p)r}{2}n^4 + (p^2-3p + r)n^2 + 2\,\, .$$
Let 
$$W \subset \bP := \bP^2 \times \bP^2 \times \bP^1$$
be a smooth hypersurface of multi-degree $(3,3,2)$, i.e., 
$W$ is a smooth $4$-fold such that 
$W \in \vert \sO_{\bP}(3,3,2) \vert$. Such $W$ certainly exists and 
it is a Calabi-Yau $4$-fold. Let $x$, $y$, $z$ be positive integers 
and consider the line bundle (invertible sheaf) on $W$:
$$L_{x, y, z} := \sO_{\bP}(x, y, z) \otimes_{\sO_{\bP}} \sO_W\, .$$
This is ample on $W$. By the exact sequence
$$0 \lra \sO_{\bP}(nx-3, ny-3, nz-2) \lra \sO_{\bP}(nx, ny, nz) \lra 
L_{x,y,z}^{\otimes n} 
\lra 0\,\,$$
it follows that
$$P(W, L_{x,y,z}) = \chi(\sO_{\bP}(nx, ny, nz)) - 
\chi(\sO_{\bP}(nx-3, ny-3, nz-2))\,\, .$$
By applying the Kunneth formula to the right hand side, it is equal to:
$$\frac{(nx+2)(nx+1)}{2} \cdot \frac{(ny+2)(ny+1)}{2} \cdot (nz+1) - \frac{(nx-1)(nx-2)}{2} \cdot \frac{(ny-1)(ny-2)}{2} \cdot (nz-1)\,\, ,$$
whence, by elementary calculations, we obtain
$$P(W, L_{x,y,z}) = \frac{(3xy(x+y)z + x^2y^2)n^4 + (6(x+y)z + 2x^2 + 9xy + 2y^2)n^2 + 4}{2}\,\,.$$
By the formulas for $P(T_k, L_k)$ and $P(W, L_{x,y,z})$, it follows that 
$T_k$ and $W$ are fibers of a flat projective morphism $f ; X \lra B$ 
over a connected base $B$, {\it if there are poitive integers} 
$$x > 0\, ,\, y > 0\, ,\, z > 0\, ,\, r > 0\, {\rm and}\, p \ge 4$$ 
{\it which satisfy the following system of equations}:
$$(p^2 - 3p)\cdot r = 3xy(x+y)z + x^2y^2\,\, ,$$
$$p^2 - 3p + r = \frac{6(x+y)z + 2x^2 + 9xy + 2y^2}{2}\,\, .$$
Though it might be possible to find all the solutions, it seems not 
so obvious. Instead of doing so, we try to find only 
(very special) solutions. For our purpose, this will be enough. 
So, we may put $y = 2x$ and $z = x$. Then 
$$3xy(x+y)z + x^2y^2 = 18x^4 + 4x^4 = 22x^4$$
$$\frac{6(x+y)z + 2x^2 + 9xy + 2y^2}{2} = \frac{18x^2 + 2x^2 + 18x^2 +8x^2}{2} 
= \frac{46x^2}{2} = 23x^2$$
and therefore the system of equations is simplified as
$$(p^2 -3p)\cdot r = 22x^4\,\, ,\,\, (p^2 -3p) + r = 23x^2\,\, .$$
Note that 
$$22x^4 = x^2 \cdot 22x^2\,\, {\rm and}\,\, 23x^2 = x^2 + 22x^2\,\, .$$ 
This means that the solutions of the following system of equations (if exist) 
give (some of) solutions of the original one:
$$p^2 - 3p = x^2\,\, ,\,\, r = 22x^2\,\, .$$
For our purpose, we may put $x = 2$. Then the last system of equations becomes:
$$p^2 - 3p = p(p-3) = 4\,\, ,\,\, r =88\,\, .$$
Thus 
$$p = 4\, ,\, r = 88\, ,\, x = 2\, ,\, y = 2x = 4\, ,\, z = x = 2$$ 
give (one of) desired solutions. 
This completes the proof of Theorem (\ref{cy}) (2).

\vskip .2cm \noindent Keiji Oguiso \\ 
Department of Mathematics\\
Osaka University\\ 
Toyonaka 560-0043 Osaka, Japan\\
oguiso@math.sci.osaka-u.ac.jp

\end{document}